\documentclass[a4paper,12pt,reqno]{amsart}
\usepackage{amssymb}
\usepackage{amsmath}
\usepackage{mathtools}
\usepackage{enumitem}
\usepackage{mathrsfs}
\usepackage[all]{xy}
\usepackage{mathtools}
\usepackage{hyperref}
\usepackage{url}
\usepackage{bbm}

\usepackage[british]{babel}
\usepackage[margin=1.2in]{geometry}
\numberwithin{equation}{section} \numberwithin{figure}{section}

\DeclareMathOperator{\Pic}{Pic} 
\DeclareMathOperator{\Gal}{Gal}

\newcommand{\OO}{\mathcal{O}}

\newcommand{\Esix}{{\mathbf E}_6}
\newcommand{\Eseven}{{\mathbf E}_7}

\newcommand{\E}{{\mathbf E}}

\newcommand{\Fr}{\ensuremath{\mathbf{F}}}
\newcommand{\kka}{\ensuremath{\overline{\Bbbk}}}

\newcommand\FF{\mathbb{F}}
\newcommand\PP{\mathbb{P}}

\newcommand\QQ{\mathbb{Q}}

\newtheorem{lemma}{Lemma}

\newtheorem{theorem}[lemma]{Theorem}

\numberwithin{table}{section}

\theoremstyle{definition}

\newtheorem{remark}[lemma]{Remark}

\newtheorem*{question}{Question}

\numberwithin{lemma}{section}

\usepackage[usenames, dvipsnames]{color}
\newcommand{\dan}[1]{{\color{blue} \sf $\clubsuit\clubsuit\clubsuit$ Dan: [#1]}}

\begin{document}

\title
{Inverse Galois problem for del Pezzo surfaces over finite fields}

\author{Daniel Loughran}
\address{Daniel Loughran \\
School of Mathematics \\
University of Manchester \\
Oxford Road \\
Manchester \\
M13 9PL \\
UK.}
\email{daniel.loughran@manchester.ac.uk}
\urladdr{https://sites.google.com/site/danielloughran/}

\author{Andrey Trepalin}
\address{Andrey Trepalin \\
Institute for Information Transmission Problems, 19 Bolshoy Karetnyi side-str., Moscow 127994, Russia}
\address{Laboratory of Algebraic Geometry, National Research University Higher School of Economics, 6 Usacheva str., Moscow 119048, Russia}
\email{trepalin@mccme.ru}


\begin{abstract}
We completely solve the inverse Galois problem for del Pezzo surfaces of degree $2$ and $3$ over all finite fields.
\end{abstract}

\dedicatory{Dedicated to the memory of Sir Henry Peter Francis Swinnerton-Dyer}

\maketitle


\section{Introduction}
Recall that a \emph{del Pezzo surface} is a smooth projective surface $S$ over a field $k$ with ample anticanonical class $-K_S$. Its \emph{degree} is defined to be the integer $d = K_S^2$. One has $1 \leq d \leq 9$, with del Pezzo surfaces of lower degree generally having more complicated arithmetic and geometry. The case $d= 3$ is by far the most famous; such surfaces are exactly cubic surfaces. See \cite[\S IV]{Man86} and \cite[\S8]{Dol12} for further background on del Pezzo surfaces.

Let $S$ be a del Pezzo surface of degree $d \leq 6$ over a finite field $\FF_q$. Over the algebraic closure $S$ is the blow-up of $\PP^2$ in $9-d$  points in general position. But $S$ can be non-rational over~$\FF_q$. To understand the geometry of $S$, one uses the action of the absolute Galois group $\Gal(\overline{\FF}_q/\FF_q)$ on the Picard group $\Pic \overline{S}$. This action preserves the canonical class and the intersection pairing, and the group of such automorphisms of $\Pic \overline{S}$ is isomorphic to the Weyl group $W\left(\mathbf{E}_{9-d} \right)$ \cite[Thm.~23.9]{Man86} (we follow the conventions of \cite[\S8.2.2]{Dol12} regarding the root systems $\mathbf{E}_{9-d}$). We therefore  obtain a well-defined conjugacy class $C(S) \subset W\left(\mathbf{E}_{9-d} \right)$, induced by the image of the Frobenius element. The surfaces with trivial conjugacy class are exactly those that are blow ups of $\PP^2$ in  rational points.
A natural problem is:

\begin{question}
Given $d$ and $q$, which conjugacy classes of $W\left(\mathbf{E}_{9-d} \right)$ can arise this way?
\end{question}

We call this the \emph{inverse Galois problem for del Pezzo surfaces over finite fields}. An analogous problem for cubic surfaces over $\QQ$ has recently been completely \mbox{solved in~\cite[Theorem 0.1]{EJ15},} where it was shown that every conjugacy class of subgroups of $W(\E_6)$ indeed arises for some cubic surface over $\QQ$.

However \emph{not} every conjugacy class of $W(\E_6)$ occurs over every finite field. For example, a simple point count shows that there is no split cubic surface over $\FF_2$ (corresponding to the trivial conjugacy class), as such a surface would have more rational points than the ambient projective space $\PP^3_{\FF_2}$! (See \cite[\S2.3.3]{Ser12}).

In \cite[Thm.~1.7]{BFL16} this question was completely answered for \emph{all large}~$q$. Specifically, the density of such surfaces was explicitly calculated for $d \leq 3$, showing that
\begin{equation} \label{eqn:dens}
\lim_{q \to \infty}\frac{\#\{ S \in \mathcal{S}_d(\FF_q) : C(S) = C\}}{\#\mathcal{S}_d(\FF_q)} = \frac{\#C}{\#W(\E_{9-d})},
\end{equation}
where $\mathcal{S}_d(\FF_q)$ denotes the  set of isomorphism classes of del Pezzo surfaces of degree~$d$ over $\FF_q$. For $d > 3$ a similar formula holds, but one needs to weight each surface by the reciprocal of the size of its automorphism group.

Unfortunately the result \eqref{eqn:dens} is not effective. One therefore still needs to determine the existence of each class, with particular attention to small finite fields.

There has been a long string of papers working on this problem solving various cases (see for example \cite{BFL16, KR16, Ryb05, RT16, SD10, Tre16, Tre18}), using a range of methods:
\begin{itemize}
	\item blowing-up suitable collections of closed points;
	\item conic bundles;
	\item Galois twist by an automorphism.
\end{itemize}
The last method is very useful for del Pezzo surfaces of degrees $2$ and~$1$, as there are always non-trivial automorphisms given by the Geiser and Bertini involution, respectively.
(See \cite[\S4.1.2]{BFL16} and \cite[\S5.1.1]{BFL16} for explicit descriptions of these twists).

For $d=5,6$ there is a unique surface over the algebraic closure, and one constructs all classes over every $\FF_q$ by using twists. The case $d=4$ was finished in \cite{Tre18}. The most general results for $d=3,2$ were also obtained in \cite{Tre18}. But there remains the incomplete case $C_{14}$ for cubic surfaces (see \cite[Thm.~1.3]{Tre18}), and four incomplete cases $28$, $35$, $47$ and $56$ for del Pezzo surfaces of degree $2$ (see \cite[Thm.~1.2]{Tre18}).

In this paper, we finish off these remaining cases and give a complete classification for cubic surfaces and del Pezzo surfaces of degree $2$.

\subsection{Cubic surfaces}

The  remaining case for cubic surfaces is conjugacy class~$C_{14}$ in Swinnerton-Dyer's notation \cite[Tab.~1]{SD67}.
In \cite[Lem.~3.5 and  \S5.4]{RT16}, this class was constructed for $q \equiv 1 \bmod 6$. The construction was completely explicit; namely
$$x_0^3 = f(x_1,x_2,x_3)$$
such that  the plane cubic $f(x_1,x_2,x_3) = 0$ is sufficiently general.


However this class has resisted for other $q$.
The main difficulty is that such surfaces are minimal, and there is no non-trivial third root of unity in $\FF_q$ for \mbox{$q \equiv 2 \bmod 3$}. Moreover, computational evidence suggests that for \mbox{$q \equiv 2 \bmod 3$} there is no such surface over $\FF_q$ with non-trivial automorphism group. So the aforementioned methods of blow-ups, twists, and conic bundles, yield nothing useful.

Despite being the most difficult case, it is one of the most common conjugacy classes; in the limit as $q \to \infty$, the density of such cubic surfaces in \eqref{eqn:dens} is $1/9$, so that one expects to find such a surface very quickly. 
We show that this class in fact exists over all finite fields, which leads to the following complete classification.

We say that a  class $C_n$ of $W(\Esix)$ \emph{exists over $\FF_q$} if there is a smooth cubic surface over $\FF_q$ whose image of the Frobenius element in $W(\Esix)$ is an element of the conjugacy class $C_n$ appearing in \cite[Tab.~1]{SD67}.

\begin{theorem} \label{thm:cubic}

Each class $C_n$ of $W(\Esix)$ exists over $\FF_q$, except in exactly the following cases.

\begin{enumerate}
\item Class $C_1$ does not exist over $\FF_2$, $\FF_3$, $\FF_5$.

\item Classes $C_2$, $C_3$ do not exist over $\FF_2$, $\FF_3$.

\item Classes $C_4$, $C_5$, $C_9$, $C_{10}$ do not exist over $\FF_2$.
\end{enumerate}
In particular, classes $C_6$--$C_8$ and $C_{11}$--$C_{25}$ exist over all finite fields.
\end{theorem}
\noindent

An expanded version of Swinnerton-Dyer's table \cite[Tab.~1]{SD67} can be found in Manin's book \cite[\S 31]{Man86}. (Swinnerton-Dyer's numbering is column $0$ of Manin's table.) Note however that Manin's table contains some mistakes; a corrected version can be found in \cite[Tab.~7.1]{BFL16}. 

\subsection{Del Pezzo surfaces of degree $2$}
The classes $47$ and $56$ can be obtained using blow-ups and twists of the last missing cubic surface  class $C_{14}$. The existence of the classes $28$ and $35$ was only unknown over $\FF_3$; we give an explicit construction of these, and thus obtain following complete classification.

We say that a class $n$ of $W(\Eseven)$ \emph{exists over $\FF_q$} if there is a smooth del Pezzo surface of degree $2$ over $\FF_q$ whose image of the Frobenius element in $W(\Eseven)$ is an element of the conjugacy class appearing in the $n$-th row of \cite[Appendix, Tab.~$1$]{Tre16}.

\begin{theorem} \label{thm:DP2}

Each class $n$ of $W(\Eseven)$ exists over $\FF_q$, except in exactly the following cases.

\begin{enumerate}
\item Classes $1$ and $49$ do not exist over $\FF_2$, $\FF_3$, $\FF_4$, $\FF_5$, $\FF_7$, $\FF_8$.

\item Classes $5$ and $10$ do not exist over $\FF_2$, $\FF_3$, $\FF_4$, $\FF_5$.

\item Classes $2$, $3$, $18$ and $31$ do not exist over $\FF_2$, $\FF_3$, $\FF_4$.

\item Classes $4$, $6$--$9$, $12$--$14$, $17$, $21$, $22$, $25$, $28$, $32$, $33$, $35$, $38$, $40$, $50$, $53$, $55$, $60$ do not exist over $\FF_2$.

\end{enumerate}
In particular, classes $11$, $15$, $16$, $19$, $20$, $23$, $24$, $26$, $27$, $29$, $30$, $34$, $36$, $37$, $39$, $41$--$48$, $51$, $52$, $54$, $56$--$59$ exist over all finite fields.

\end{theorem}

The table \cite[Appendix, Tab.~$1$]{Tre16} is based upon Carter's \cite[Tab. $10$]{Car72}. Urabe \cite[Tab.~1]{Ura96a} also compiled a table of the possibilities for del Pezzo surfaces of degree $2$; but note that the rows in Urabe's table are ordered differently to Carter's table.

Knowledge about the possibilities for the Galois action of del Pezzo surfaces over small finite fields is often required in proofs of some results concerning properties of del Pezzo surfaces (see e.g.~the proof of \cite[Thm.~1]{STVA14}). We hope that the results in this paper should assist with such arguments in the future.

\subsection*{Acknowledgements} We thank Costya Shramov and Yuri Prokhorov for organising the ``Workshop on birational geometry'' at the NRU HSE, Moscow, in October 2018, where this work was commenced. We also thank Sergey Rybakov for useful discussions about the remaining case $C_{14}$, and the referee for helpful comments. The first-named author is supported by EPSRC grant EP/R021422/1. The second-named author is supported by the Russian Academic Excellence Project '5--100' and Young Russian Mathematics award.

\section{Cubic surfaces --- Proof of Theorem \ref{thm:cubic}}
As explained in the introduction, by \cite[Thm.~1.3]{Tre18} it suffices to construct the remaining case $C_{14}$ over every finite field. To do so we require the following lemma.

\begin{lemma}
\label{cubiclemma}
Let $k$ be an algebraically closed field and let $\Pi_1$, $\Pi_2$ and $\Pi_3$ be three planes in~$\PP^3_{k}$ which have exactly one common point $Q$. Let $C_1$, $C_2$ and $C_3$ be three cubic curves lying in $\Pi_1$, $\Pi_2$ and $\Pi_3$ respectively, such that $\Pi_i \cap C_j = \Pi_j \cap C_i$ and such that  $Q \notin C_i$ for any $i,j$.  Then there is one-dimensional family of cubic surfaces over $k$ containing $C_1$, $C_2$ and $C_3$.

\end{lemma}

\begin{proof}
We can assume the planes are given by $\Pi_i : x_i= 0$, so $Q=(1:0:0:0)$. The cubic curves are given by $C_i: P_i(x_0,\ldots, \widehat{x_i}, \dots, x_3) = 0$. Each polynomial $P_i$ contains the monomial $x_0^3$, since $Q \notin C_i$, so we may assume that the coefficient of~$x_0^3$ is $1$ for each polynomial $P_i$. 
Put
\begin{align*}
P(x_0,x_1, x_2,x_3) &= P_1(x_0,x_2,x_3) + P_2(x_0,x_1,x_3) + P_3(x_0,x_1,x_2) \\
 &- P_1(x_0,0,x_3) - P_2(x_0,x_1,0) - P_3(x_0,0,x_2) + x_0^3.
\end{align*}
Then one has
$$
P(x_0,0,x_2,x_3) = P_1(x_0,x_2,x_3), \qquad P(x_0,x_1,0,x_3) = P_2(x_0,x_1,x_3),
$$
$$
P(x_0,x_1,x_2,0) = P_3(x_0,x_1,x_2),
$$
\noindent since $P_1(x_0,0,0) = P_2(x_0,0,0) = P_3(x_0,0,0) = x_0^3$ and
$$
P_1(x_0,0,x_3)= P_2(x_0, 0, x_3), P_1(x_0,x_2,0) = P_3(x_0,0,x_2), P_2(x_0,x_1,0) = P_3(x_0,x_1,0).
$$
Therefore for any $(\mu : \lambda) \in \PP^1(k)$ the equation $\mu x_1x_2x_3 = \lambda P(x_0,x_1,x_2,x_3)$ gives a cubic surface passing through the curves $C_1$, $C_2$ and $C_3$.
\end{proof}

We  apply this as follows. Let $\FF_q$ be a finite field and $\Fr$ the Frobenius automorphism. Let $L_1$, $L_2$ and $L_3$ be a triple of conjugate non-coplanar lines in $\PP^3_{\FF_{q^3}}$ passing through a common  point $Q \in \PP^2(\FF_q)$. Consider a point $p_1 \in L_1(\FF_{q^9}) \setminus L_1(\FF_{q^3})$, and its conjugates $p_{i + 1} = \Fr^i p_1$, $i \in \{2, \ldots, 9\}$. 
One has $p_{i}, p_{i+3}, p_{i+6} \in L_i$ (here and in what follows,  subscripts are taken modulo $9$).
 Consider the line $E_1$ passing through $p_1$ and $p_2$, and its conjugates $E_{i + 1} = \Fr^i E_1$, $i \in \{2, \ldots, 9\}$. Denote this configuration of nine conjugate lines by $\mathcal{E}$. We want to show that there exists a smooth cubic surface $X$ over $\FF_q$ containing $\mathcal{E}$.

The configuration $\mathcal{E}$ is defined over $\FF_q$, in particular the linear system of cubics containing it is defined over $\FF_q$.
For $\{i,j,k\} = \{1,2,3\}$, let $\Pi_i$ be the plane spanned by $L_j$ and $L_k$.
Then $E_{i+1}, E_{i+4}, E_{i+7} \subset \Pi_i$ and $\Pi_i \cap \Pi_j \cap \mathcal{E} = \{p_{k}, p_{k+3}, p_{k+6}\}$.
 Therefore by Lemma \ref{cubiclemma} there is one-dimensional family $\mathcal{X}$ of cubic surfaces over~$\FF_q$ passing through $\mathcal{E}$. We now show that there is a unique singular member.

\begin{lemma}
\label{cubicsing}
There is a unique singular cubic surface over $\FF_q$ containing the configuration~$\mathcal{E}$. This surface is the union of the planes $\Pi_1$, $\Pi_2$ and $\Pi_3$.
\end{lemma}

\begin{proof}
Let $X$ be a singular surface containing $\mathcal{E}$, and $S$ a singular point of $X_{\overline{\FF}_q}$. There is a line $E$ in $\mathcal{E}$ that does not pass through $S$. Without loss of generality, we may assume that $E=E_1$. Then $E_1$ meets the other lines in $\mathcal{E}$ in at least $3$ points, as it meets the lines $E_2$, $E_9$ and $E_4$ in distinct points.
Thus any line in~$\PP^2$ passing through one of these points and $S$ meets both points in multiplicity at least~$2$, hence is contained in $X$. Thus the plane spanned by $E$ and $S$ contains at least~$4$ lines in~$X$, therefore this plane lies in $X$. 
As the lines $E_i$ are conjugate, the only possibility is that $X$ is the union of planes of planes $\Pi_1$, $\Pi_2$ and $\Pi_3$.
\end{proof}


So by Lemmas \ref{cubiclemma} and \ref{cubicsing} there exists a smooth cubic surface $X$ containing the configuration $\mathcal{E}$, as $\mathcal{X}$ has $q+1 \geq 3$ elements. Any such cubic surface contains nine conjugate lines. Therefore it splits after a field extension of degree at least~$9$, hence its conjugacy class in $W(\Esix)$ consists of elements of order divisible by $9$. However from the classification \cite[Tab.~1]{SD67}, one can see that $C_{14}$ is the unique such class (in fact its elements have order exactly $9$). Thus $X$ has the claimed class. \qed

\section{Del Pezzo surfaces of degree $2$ --- Proof of Theorem \ref{thm:DP2}}
By Theorem \ref{thm:cubic}, over every finite field $\FF_q$ there exists a cubic surface $S$ with class~$C_{14}$. But every cubic surface over a finite field has a rational point, by the Chevalley--Warning theorem \cite[\S2.2]{Ser73}. Moreover, by a consideration of the Galois action on the lines, no rational point lies on a line over $\overline{\FF}_q$. Thus the blow-up of $S$ in a rational point  is a del Pezzo surface of degree~$2$. This constructs the missing class $47$ over $\FF_q$, since this class corresponds to the unique subgroup of order $9$ in~$W(\Eseven)$. One then performs a Geiser twist to get the class $56$ (see \cite[Table 1]{Tre18}).

By \cite[Thm.~1.2]{Tre18}, this leaves open the existence of classes $28$ and $35$ over~$\FF_3$. These surfaces are Geiser twists of each other (see \cite[Table 1]{Tre18}), thus it suffices to construct the class~$35$. We claim that this is realised by the surface
$$X: (x^2 + xz - z^2)s^{2} + (x^2 + y^2)st + (x^2 - xy - y^2 + xz - z^2)t^{2} = 0\subset \PP^1_{s,t} \times \PP^2_{x,y,z}.$$
We first prove that $X$ is a del Pezzo surface of degree $2$. First, this surface is easily checked to be smooth over $\FF_3$. By the adjunction formula the anticanonical bundle is $\OO_X(0,1)$, thus the anticanonical map is given by the projection onto the second factor $X \to \PP^2$. This map is proper and a simple calculation shows that it is quasi-finite, hence it is finite. Thus the anticanonical bundle, being the pull-back of an ample line bundle by a finite morphism, is ample. Thus $X$ is del Pezzo, and from the equation it is obvious that it has degree $2$.

Next, the projection onto the first factor equips $X$ with the structure of a conic bundle. The singular fibres lie over the closed points of $\PP^1$ given by $s=0$, $t= 0$, $s^2 + t^2 = 0$ and $s^2 + st - t^2= 0$, moreover one checks that these singular fibres are all irreducible, so that $X \to \PP^1$ is relatively minimal. Therefore from \cite[Thm.~2.6]{Ryb05}, one finds that the Frobenius element acts on the Picard group $\Pic \overline{X}$ with eigenvalues $1,-1,-1,i,-i,i,-i$. However, an inspection of \cite[Appendix, Tab.~$1$]{Tre16} reveals that class $35$ is the only one with this property, hence the class is $35$, as claimed. \qed

\begin{remark}
In fact any del Pezzo surface of degree $2$ with a conic bundle structure may be embedded as a surface of bidegree $(2,2)$ in $\PP^1 \times \PP^2$; see \cite[Thms.~5.6, 5.7]{FLS18}.
\end{remark}

\end{document}